\documentclass[10pt,reqno]{amsart}
\usepackage[margin=1.1in]{geometry}

\newcommand{\NN}{\mathbb{N}}
\newcommand{\RR}{\mathbb{R}}
\newcommand{\ZZ}{\mathbb{Z}}

\newcommand{\Hom}{\mathrm{Hom}}

\newcommand{\Spec}{\mathrm{Spec}}

\newcommand{\rank}{\mathrm{rank}}

\newcommand{\rel}{\mathrm{rel}}
\newcommand{\abs}{\mathrm{abs}}
\newcommand{\can}{\mathrm{can}}
\newcommand{\Pic}{\mathrm{Pic}}

\usepackage{amsthm}
\usepackage{cite}

\newtheorem{thm}{Theorem}[section]

\newtheorem{prop}[thm]{Proposition}
\newtheorem{lem}[thm]{Lemma}
\newtheorem{claim}[thm]{Claim}
\newtheorem{cor}[thm]{Corollary}

\newtheorem*{thm-non}{Theorem}

\theoremstyle{definition}
\newtheorem{defn}[thm]{Definition}

\newtheorem{rem}[thm]{Remark}

\newtheorem{set-up}[thm]{Set-up}

\usepackage{graphicx}
\usepackage{amssymb}
\usepackage{epstopdf}
\usepackage{enumerate}
\usepackage{tikz}
\usepackage{mathtools}

\input xy
\xyoption{all}
\numberwithin{equation}{section}

\begin{document}

\title{MAXIMALLY FROBENIUS-DESTABILIZED VECTOR BUNDLES OVER SMOOTH ALGEBRAIC CURVES}

\author{YIFEI ZHAO}

\address{Department of Mathematics, Columbia University, New York, NY 10027 \\
yz2427@columbia.edu}

\maketitle

\begin{abstract}
Vector bundles in positive characteristics have a tendency to be destabilized after pulling back by the Frobenius morphism. In this paper, we closely examine vector bundles over curves that are, in an appropriate sense, maximally destabilized by the Frobenius morphism. Then we prove that such bundles of rank 2 exist over any curve in characteristic 3, and are unique up to twisting by a line bundle. We also give an application of such bundles to the study of ample vector bundles, which is valid in all characteristics.
\end{abstract}

\setcounter{section}{-1}
\section{Introduction}
Given a normal projective variety $X$ over an algebraically closed field $k$ of characteristic $p\neq 0$, with a fixed ample divisor, it can happen that pulling back by an inseparable morphism $f:Y\rightarrow X$ destroys semistability of vector bundles over $X$. In the simplest case, where $X$ is a smooth curve and $f=F_{\rel}$ is the (relative) Frobenius morphism, semistable vector bundles whose pullback under $f$ fail to be semistable are called \emph{Frobenius-destabilized} vector bundles. They are closely related to the study of the generalized Verschiebung.

In this paper, we are primarily interested in rank-$r$ vector bundles $E$ over a curve $X$ of genus $g\ge 2$ with the following property:
\begin{enumerate}[($*$)]
	\item the Harder-Narasimhan filtration of $F_{\rel}^*E$ admits line bundle quotients $L_1,\cdots, L_r$, and $\deg(L_i)=\deg(L_{i+1})+2g-2$ for all $i=1,\cdots, r-1$.
\end{enumerate}
If themselves semistable, such bundles are the ``most destabilized" ones possible. We will call a semistable vector bundle $E$ with property ($*$) a \emph{maximally Frobenius-destabilized} vector bundle (cf.~Def.~\ref{defn-max-destab}). These vector bundles appear in the works of many authors. Notably, over a genus-$2$ curve, all rank-$2$ Frobenius-destabilized vector bundles are maximally Frobenius-destabilized (cf.~\cite[Prop.~3.3]{Jos-Xia}); the work of S.~Mochizuki \cite{Moc} gave a precise formula counting the number of such bundles with trivial determinant over a general curve in arbitrary characteristic, and B.~Osserman \cite{Oss} counted them over an \emph{arbitrary} curve in small characteristics. In another direction, K.~Joshi \emph{et al.~}\cite{Jos} gave a relation between certain Frobenius-destabilized bundles and pre-opers, a concept originated from the geometric Langlands program; their observation was later used by X.~Sun \cite{Sun} to prove that stability is preserved under Frobenius-pushforward.

We prove that vector bundles with property ($*$) exhibit an interesting trichotomy:
\begin{thm-non}
Let $E$ be a rank-$r$ vector bundle over a curve $X$ of genus $g\ge 2$ with property ($*$). Then
\begin{enumerate}[(i)]
	\item if $r<p$ and $p$ does not divide $g-1$, then $E$ is stable (Prop.~\ref{prop-r<p});
	\item if $r>p$, then $E$ is not semistable (Prop.~\ref{prop-r>p}); and
	\item if $r=p$, then $E$ is stable if and only if $E$ is the Frobenius-pushforward of a line bundle (Prop.~\ref{prop-r=p}; this is already discovered by L.-G.~Li, H.~Lange and C.~Pauly).
\end{enumerate}
\end{thm-non}
\noindent
The result (i) reflects a general observation of K.~Joshi \emph{et al.~}that high instability of the Frobenius-pullback $F_{\rel}^*E$ implies high stability of $E$. The main technical ingredient in the proof of (i) is an improvement of an inequality due to N.~Shepherd-Barron \cite[Cor.~$2^p$]{She}: $L_{\max}(E)-L_{\min}(E)\le (r-1)(2g-2)/p$. Here $L_{\max}$ and $L_{\min}$ are measures of maximal slope of subbundles and minimal slope of quotient bundles respectively, taken over all finite pullbacks. We split Shepherd-Barron's inequality into
\begin{equation}
\label{eq-improved-ineq}
L_{\max}(E)-\mu(E),\quad\mu(E)-L_{\min}(E)\le (r-1)(g-1)/p
\end{equation}
as well as giving an equality criterion for both. As we will see, the improved inequalities \eqref{eq-improved-ineq} have many independent applications as well. The result (ii) shows that maximally Frobenius-destabilized vector bundles do not exist with rank $r>p$. For the critical value $r=p$, property ($*$) is intimately related to pre-opers; the result (iii) is known (cf.~Li \cite[Lem.~3.1]{Li} and Lange and Pauly \cite[Prop.~1.2]{Lange}) and a more general version is given by C.~Liu and M.~Zhou \cite{Liu}.

In characteristic $p=3$, B.~Osserman \cite{Oss} has constructed rank-$2$ maximally Frobenius-destabilized vector bundles with trivial determinant over an arbitrary curve of genus $2$, and showed that there are exactly $\#\Pic(X)[2]=16$ of them. We derive from the previous theorem that, for \emph{arbitrary} genus, such bundles are precisely the rank-$2$ vector bundles $E$ satisfying
\begin{equation}
\label{eq-sym-isom}
S^2(E)\cong (F_{\rel})_*T
\end{equation}
where $S^2(E)$ is the second symmetric power of $E$, and $T$ is the tangent sheaf. Then we turn \eqref{eq-sym-isom} around, and construct vector bundle $E$ whose second symmetric power is isomorphic to $(F_{\rel})_*T$. Our main result in this direction is
\begin{thm-non}[Thm.~\ref{thm-trivial-det}]
Over any curve $X$ of genus $g\ge 2$ in characteristic $3$, there exist precisely $\#\Pic(X)[2]=2^{2g}$ rank-$2$ maximally Frobenius-destabilized vector bundles of trivial determinant, and they are related to one another via twisting by an order-$2$ line bundle.
\end{thm-non}
\noindent

Y.~Wakabayashi \cite[Thm.~A]{Wak} recently gave a formula counting the number of rank-$2$ Frobenius-destabilized bundles on a general curve of genus $g$, subject to the condition $2(g-1)<p$. In contrast, our result holds for $p=3$ and arbitrary genus.

In characteristic zero, a vector bundle $E$ over a curve is ample if and only if the minimal slope of vector bundle quotients of $E$ is positive (cf.~\cite[Thm.~3.2.7]{Huy}). In positive characteristics, a numerical criterion for ample vector bundles over curves is given by H.~Brenner \cite[Thm.~2.3]{Bre} and I.~Biswas \cite{Bis}. We generalize their result to higher dimensions:
\begin{thm-non}[Thm.~\ref{thm-ample-criterion-text}]
Let $X$ be a smooth projective variety over an algebraically closed field $k$, and $E$ be a vector bundle over $X$. Then $E$ is ample if and only if there exists some real number $\varepsilon>0$, such that for any integral, closed curve $C$ in $X$, there holds $L_{\min}(E|_C)\ge\varepsilon\|C\|$.
\end{thm-non}
\noindent
Here, the norm $\|\cdot\|$ on $A_1(X)$ is fixed with respect to a chosen basis.

The improved inequalities \eqref{eq-improved-ineq} together with the above ample criterion show that semistable vector bundles with degree greater than $r(r-1)(g-1)/p$ over nonrational curves are ample. On the other hand, maximally Frobenius-destabilized vector bundles with degree less than $r(r-1)(g-1)/p$ are \emph{not} ample (Cor.~\ref{cor-deg-criterion}). We use this last observation to construct certain non-ample semistable vector bundles of positive degree.

\medskip

\noindent
\textit{Acknowledgements.} First and foremost, the author thanks A.~J.~de Jong for supervising this project and for many helpful conversations. He is also grateful to A.~Langer and H.~Brenner for suggesting many additional references, and to the anonymous referee for pointing out an error in an earlier version of this paper.

\section{Preliminaries}

\label{sec-notation} We first fix some notations. Let $X$ be a normal projective variety over an algebraically closed field $k$, and let $E$ be a torsion-free coherent sheaf over $X$. Fix an ample line bundle $H$ over $X$. With respect to $H$, we set $\mu_{\max}(E)$ to be the maximal slope of a coherent subsheaf of $E$, and $\mu_{\min}(E)$ to be the minimal slope of a torsion-free coherent quotient sheaf of $E$. If $0=E_0\subset E_1\subset\cdots\subset E_l=E$ is the Harder-Narasimhan filtration of $E$ with successive quotients $Q_i=E_i/E_{i-1}$, then
$$
\mu_{\max}(E)=\mu_1(E)>\mu_2(E)>\cdots>\mu_l(E)=\mu_{\min}(E)
$$
where $\mu_i(E)=\mu(Q_i)$ is the slope of $Q_i$. With respect to the ample line bundles pulled back from $H$, we also set
\begin{equation}
L_{\max}(E)=\sup_{f:Y\rightarrow X}\frac{\mu_{\max}(f^*E)}{\deg(f)},\quad\text{and}\quad L_{\min}(E)=\inf_{f:Y\rightarrow X}\frac{\mu_{\min}(f^*E)}{\deg(f)}
\end{equation}
where both the supremum and the infimum are taken over all finite morphisms $f:Y\rightarrow X$ of normal projective varieties over $k$. We note the following standard

\begin{lem}
\label{lem-sep-hn}
Let $f:Y\rightarrow X$ be a finite, separable morphism of normal projective varieties. Let $E$ be a torsion-free coherent sheaf over $X$ with Harder-Narasimhan filtration $0=E_0\subset E_1\subset\cdots\subset E_l=E$. Then $0=f^*E_0\subset f^*E_1\subset\cdots\subset f^*E_l=f^*E$ is the Harder-Narasimhan filtration of $f^*E$.
\end{lem}
\begin{proof}[Sketch of proof]
Using an argument similar to that of \cite[Lem.~3.2.2]{Huy}, one first proves that if $E_1$ is the maximal destabilizing subsheaf of $E$, then $f^*E_1$ is the maximal destabilizing subsheaf of $f^*E$. Apply this fact with $E$ replaced by $E/E_{i}$ for all $i=1,\cdots, l-1$, and the result follows.
\end{proof}

\begin{cor}
If the ground field $k$ has characteristic zero, then $L_{\max}(E)=\mu_{\max}(E)$, and $L_{\min}(E)=\mu_{\min}(E)$.
\end{cor}
\begin{proof}
In this case, every finite morphism $f:Y\rightarrow X$ of normal projective varieties is separable. Hence the lemma shows that
$$
\frac{\mu_{\max}(f^*E)}{\deg(f)}=\mu_{\max}(E),\quad\text{and}\quad\frac{\mu_{\min}(f^*E)}{\deg(f)}=\mu_{\min}(E)
$$
The corollary then follows from the definitions of $L_{\max}(E)$ and $L_{\min}(E)$.
\end{proof}

\noindent
If the ground field $k$ has characteristic $p\neq 0$, we may let $F_{\abs}:X\rightarrow X$ be the absolute Frobenius morphism, which factors through the relative Frobenius morphism $F_{\rel}:X\rightarrow X^{(1)}$ as in the following commutative diagram:
$$
\xymatrix{X\ar@/^/[rrd]^{F_{\abs}} \ar[rd]^{F_{\rel}} \ar@/_/[rdd] & & \\
& X^{(1)}\ar[r]\ar[d] & X\ar[d] \\
& \Spec(k) \ar[r]_{a\mapsto a^p} & \Spec(k)}
$$
Here, the commutative square is a pullback diagram.

\begin{cor}
\label{cor-lmax-lmin-charp}
If the ground field $k$ has characteristic $p\neq 0$, then
\begin{equation}
\label{eq-lmax-lmin}
L_{\max}(E)=\lim_{k\rightarrow\infty}\frac{\mu_{\max}((F_{\rel}^k)^*E)}{p^k},\quad\text{and}\quad L_{\min}(E)=\lim_{k\rightarrow\infty}\frac{\mu_{\min}((F_{\rel}^k)^*E)}{p^k}
\end{equation}
\end{cor}
\noindent
Observe that we may replace $F_{\rel}$ by $F_{\abs}$ in \eqref{eq-lmax-lmin} since $X$ and $X^{(1)}$ are isomorphic as schemes. This corollary shows that our definitions of $L_{\max}(E)$ and $L_{\min}(E)$ agree with those in \cite{Lan}.
\begin{proof}
Note that the sequence $\mu_{\max}((F_{\rel}^k)^*E)/p^k$ is nondecreasing. We only need to show for every finite morphism $f:Y\rightarrow X$ of normal projective varieties, there exists some $k\in\mathbb{N}$ such that
$$
\frac{\mu_{\max}(f^*E)}{\deg(f)}\le\frac{\mu_{\max}((F_{\rel}^k)^*E)}{p^k}
$$
Indeed, given such a morphism $f:Y\rightarrow X$, upon further pulling back, we may assume the induced field extension $K(X)\subset K(Y)$ is normal. Thus we may factor $f$ into the composition of a separable morphism $g:Y\rightarrow Y'$ of normal projective varieties, and a purely inseparable morphism $h:Y'\rightarrow X$; moreover, $h$ is ``dominated" by a composition of $F_{\rel}$ as illustrated in the following commutative diagram:
$$
\xymatrix{Y\ar[r]^g\ar[dr]_f & Y' \ar[d]^h & X^{(-k)}\ar[l]\ar[dl]^{F_{\rel}^k} \\
& X &}
$$
Applying Lem.~\ref{lem-sep-hn} to morphism $g$, we obtain
$$
\frac{\mu_{\max}(f^*E)}{\deg(f)}=\frac{\mu_{\max}(h^*E)}{\deg(h)}\le\frac{\mu_{\max}((F_{\rel}^k)^*E)}{p^k}
$$
as desired. The proof for the second identity is similar.
\end{proof}

 One of the most important properties of $L_{\max}(E)$ and $L_{\min}(E)$ is that they are determined after pulling back by $F_{\rel}$ finitely many times. This result is due to A.~Langer:

\begin{thm}[A.~Langer]
\label{thm-fdHN}
Let $X$ be a smooth projective variety over an algebraically closed field of positive characteristic $p$, and let $E$ be a torsion-free coherent sheaf over $X$. Then there exists a natural number $k_0$ such that if
$$
0=E_0\subset E_1\subset\cdots \subset E_l=(F_{\rel}^{k_0})^*E
$$
is the Harder-Narasimhan filtration of $(F_{\rel}^{k_0})^*E$, then for every $k\ge k_0$,
$$
0=(F_{\rel}^{k-k_0})^*E_0\subset (F_{\rel}^{k-k_0})^*E_1\subset\cdots\subset (F_{\rel}^{k-k_0})^*E_l=(F_{\rel}^k)^*E
$$
is the Harder-Narasimhan filtration of $(F_{\rel}^k)^*E$. In other words, the successive quotients $Q_i=E_i/E_{i-1}$ are strongly semistable.
\end{thm}
\begin{proof}
This is \cite[Thm.~2.7]{Lan}.
\end{proof}

 We now assume the ground field $k$ has characteristic $p\neq 0$. Given a coherent sheaf $E$ over $X$, and a connection $\nabla:E\rightarrow E\otimes\Omega_X$, the $p$-curvature of $\nabla$ is the $p$-linear sheaf morphism $\mathrm{Der}_k(\mathcal{O}_X)\rightarrow\mathcal{E}\mathrm{nd}_{\mathcal{O}_X}(E)$ given by
$$
\theta\mapsto (\nabla_{\theta})^p-\nabla_{\theta^p}
$$
There is a canonical connection $\nabla_{\can}:F_{\rel}^*E\rightarrow F_{\rel}^*E\otimes\Omega_{X^{(-1)}}$, defined locally in a straightforward manner. A theorem of Cartier (see, for example, \cite[Thm.~5.1]{Kat}) shows the following equivalence of categories:
$$
\left( \begin{array}{c}
\text{coherent sheaves} \\
\text{over $X$} \end{array} \right)
\xrightarrow{E\mapsto (F_{\rel}^*E,\nabla_{\can})}
\left( \begin{array}{c}
\text{coherent sheaves on $X^{(-1)}$ with integrable} \\
\text{connections of vanishing $p$-curvature} \end{array} \right)
$$
An immediate consequence of Cartier's theorem is the following

\begin{lem}
\label{lem-pullback-connection}
Let $X$ be a smooth projective variety, and $E$ be a torsion-free coherent sheaf over $X$. Then for every coherent subsheaf $F\subset F_{\rel}^*E$, we have
\begin{enumerate}[(i)]
	\item $F$ is the pullback of some subsheaf of $E$ along $F_{\rel}$ if and only if $\nabla_{\can}(F)\subset F\otimes\Omega_{X^{(-1)}}$.
	
	\item If $F$ is not the pullback of any subsheaf of $E$ along $F_{\rel}$, then the induced map $\nabla_{\can}:F\rightarrow (F_{\rel}^*E/F)\otimes\Omega_{X^{(-1)}}$ is a nonzero $\mathcal{O}_{X^{(-1)}}$-linear morphism.
\end{enumerate}
\end{lem}
\begin{proof}
For (i), the ``only if" direction is clear. For the ``if" direction, observe that $\nabla_{\can}$, when restricted to $F$, gives an integrable connection of vanishing $p$-curvature. By Cartier's theorem, the inclusion $F\rightarrow F^*_{\rel}E$ is the pullback of some morphism $E'\rightarrow E$. On the other hand, $F_{\rel}$ is flat when $X$ is smooth, so this morphism $E'\rightarrow E$ has to be an inclusion. (ii) follows directly from (i), except that the $\mathcal{O}_{X^{(-1)}}$-linearity has to be checked locally.
\end{proof}

\section{Maximally Frobenius-destabilized vector bundles}
In this entire section, $X$ will denote a smooth projective curve over an algebraically closed field $k$ of characteristic $p\neq 0$, with genus $g\ge 1$. For notational convenience, we will denote $X^{(-1)}$ by $Y$, thus writing the relative Frobenius morphism $F_{\rel}:Y\rightarrow X$.

We start by refining the following inequality discovered by I.~N.~Shepherd-Barron \cite[Cor.~$2^p$]{She} and generalized to higher dimensions by A.~Langer \cite[Cor.~6.2]{Lan}:
\begin{equation}
\label{eq-classical-ineq}
L_{\max}(E)-L_{\min}(E)\le\frac{(r-1)(2g-2)}{p}
\end{equation}
for any rank-$r$ semistable vector bundle $E$ over $X$.

\begin{prop}
\label{prop-langer-ineq}
Let $E$ be a semistable vector bundle over $X$ of rank $r$. Then
\begin{equation}
\label{eq-langer-ineq}
L_{\max}(E)-\mu(E)\le\frac{(r-1)(g-1)}{p}\quad\text{and}\quad \mu(E)-L_{\min}(E)\le\frac{(r-1)(g-1)}{p}
\end{equation}
Furthermore, the following are equivalent, provided $g\ge 2$:
\begin{enumerate}[(i)]
	\item $L_{\max}(E)-\mu(E)=(r-1)(g-1)/p$;
	\item $\mu(E)-L_{\min}(E)=(r-1)(g-1)/p$;
	\item $F_{\rel}^*E$ satisfies the following condition:
	\begin{equation}
	\label{eq-condition}
	\begin{array}{c}
\text{Its Harder-Narasimhan filtration consists entirely of rank-one quotients} \\
\text{$L_1,\cdots,L_r$ and $\mu(L_i)=\mu(L_{i+1})+2g-2$ for all $i=1,\cdots,r$.}\end{array}
	\end{equation}
\end{enumerate}
\end{prop}

\noindent
We will see later (Prop.~\ref{prop-r>p}) that (iii) cannot be satisfied for semistable vector bundles $E$ of rank $r>p$. Hence the inequalities \eqref{eq-langer-ineq} are not sharp for $r>p$. In that case, we will supply a better inequality (Prop.~\ref{prop-langer-ineq-r>p}).

\begin{proof}
We start with a technique used by Langer in the proof of \cite[Cor.~6.2]{Lan}. By Thm.~\ref{thm-fdHN}, we may fix some sufficiently large $k\in\mathbb{N}$ such that
$$
L_{\max}(E)=\frac{\mu_{\max}((F_{\rel}^k)^*E)}{p^k}\quad\text{and}\quad L_{\min}(E)=\frac{\mu_{\min}((F_{\rel}^k)^*E)}{p^k}
$$
There is a connection
$$
\eta:(F_{\rel}^k)^*E\rightarrow(F_{\rel}^k)^*E\otimes(\Omega_{X^{(-k)}}\oplus\cdots\oplus(F_{\rel}^{k-1})^*\Omega_{Y})
$$
given by $\eta=(\nabla_{\can}, \cdots, (F^{k-1})^*\nabla_{\can})$, where $\nabla_{\can}$ is the canonical connection induced by the Frobenius morphism. Let $0=E_0\subset\cdots\subset E_l=(F_{\rel}^k)^*E$ be the Harder-Narasimhan filtration of $(F_{\rel}^k)^*E$. Since $E$ is semistable, none of the $E_i$'s are pullbacks of subbundles of $E$ by $F_{\rel}^k$. By Lem.~\ref{lem-pullback-connection}, there is a nonzero $\mathcal{O}_{X^{(-k)}}$-linear morphism for each $E_i$:
$$
E_i\rightarrow((F_{\rel}^k)^*E/E_i)\otimes(\Omega_{X^{(-k)}}\oplus\cdots\oplus(F_{\rel}^{k-1})^*\Omega_{Y})
$$
i.e.~given each $i$, there is a nonzero morphism $E_i\rightarrow((F_{\rel}^k)^*E/E_i)\otimes(F_{\rel}^j)^*\Omega_{X^{(j-k)}}$ for some $j=0,\cdots,k-1$. Thus
$$
\mu_i((F_{\rel}^k)^*E)\le\mu_{i+1}((F_{\rel}^k)^*E)+p^j(2g-2)\le\mu_{i+1}((F_{\rel}^k)^*E)+p^{k-1}(2g-2)
$$
where the second equality is attained if and only if $E_i$ is the pullback of some subbundle of $F_{\rel}^*E$ by $F_{\rel}^{k-1}$.

\begin{claim}[cf.~\cite{Jos}, Lem.~4.2.4]
Fix $\mu,\eta\in\RR$ and $r\in\NN$; assume $\eta\ge 0$. For any sequence of real numbers $\mu_1>\mu_2>\cdots>\mu_l$ with $\mu_i\le\mu_{i+1}+\eta$, and any sequence of natural numbers $r_1,\cdots,r_l$ of indeterminate length such that $r_1+\cdots+r_l=r$ and $r_1\mu_1+\cdots+r_l\mu_l=r\mu$, there hold
\begin{equation}
\label{eq-elementary-ineq}
\mu_1\le\mu+\frac{(r-1)\eta}{2} \quad\text{and}\quad\mu_l\ge\mu-\frac{(r-1)\eta}{2}
\end{equation}
Furthermore, the following are equivalent if we assume $\eta>0$:
\begin{enumerate}[(i)]
	\item $\mu_1=\mu+(r-1)\eta/2$;
	\item $\mu_l=\mu-(r-1)\eta/2$;
	\item $l=r$, $r_i=1$ for all $i$, and $\mu_i=\mu_{i+1}+\eta$.
\end{enumerate}
\end{claim}

\noindent
Assuming the claim and applying it to $\mu=\mu((F_{\rel}^k)^*E)$, $\eta=p^{k-1}(2g-2)$, $\mu_i=\mu_i((F_{\rel}^k)^*E)$, and $r_i=\rank(E_i)$, the inequalities \eqref{eq-langer-ineq} readily follow. Furthermore, in the $g\ge 2$ case, both equalities are attained when each $E_i$ is pullback of a subbundle of $F_{\rel}^*E$ by $F_{\rel}^{k-1}$, and the Harder-Narasimhan filtration of $F_{\rel}^*E$ must have rank-one quotients $L_1,\cdots,L_r$, with $\mu(L_i)=\mu(L_{i+1})+2(g-1)$.

We now prove the elementary claim. First, if $\eta=0$, then $l=1$ and $\mu=\mu_1=\mu_l$; the inequalities \eqref{eq-elementary-ineq} are trivially satisfied. We now assume $\eta>0$. Indeed, using $\mu_i\ge\mu_1-(i-1)\eta$, we find
$$
\mu=\frac{1}{r}(r_1\mu_1+\cdots+r_l\mu_l)\ge\mu_1-\frac{\eta}{r}(r_2+2r_3+\cdots+(l-1)r_l)
$$
where equality is attained if and only if $\mu_i=\mu_{i+1}+\eta$ for each $i$. If $l=1$, we obtain strict inequalities in \eqref{eq-elementary-ineq} thank to the $\eta>0$ assumption. Now, for any $l=2,\cdots,r$, the maximum of $r_2+2r_3+\cdots+(l-1)r_l$, subject to the condition $r_1+r_2+\cdots+r_l=r$, is obtained when $r_1=\cdots=r_{l-1}=1$ and $r_l=r-l+1$. Altogether,
$$
r_2+2r_3+\cdots+(l-1)r_l\le-\frac{1}{2}l^2+\left(r+\frac{1}{2}\right)l-r\le\frac{1}{2}r(r-1)
$$
where the last equality is attained when $l=r$. Summarizing these inequalities, we get
$$
\mu_1\le\mu+\frac{\eta}{r}\cdot\frac{1}{2}r(r-1)=\mu+\frac{(r-1)\eta}{2}
$$
and the equivalence of (i) and (iii). The other inequality in \eqref{eq-elementary-ineq} and the equivalence of (ii) and (iii) can be deduced in a similar manner.
\end{proof}

\begin{rem}
\label{rem-max-destab}
In the last condition (iii) in Prop.~\ref{prop-langer-ineq}, the numerical equation $\mu(L_i)=\mu(L_{i+1})+2g-2$ is equivalent to $L_i\cong L_{i+1}\otimes\Omega_Y$ (provided that $E$ is semistable). This can be seen as a consequence of the following easy
\end{rem}

\begin{lem}
\label{lem-condition-equiv}
Suppose $g\ge 2$, and $F_{\rel}^*E$ satisfies \eqref{eq-condition}, and none of the subbundles $E_i$ $(1\le i\le r-1)$ in the Harder-Narasimhan filtration of $F_{\rel}^*E$ occurs as pullback of a subbundle of $E$ along $F_{\rel}$. Then $\nabla_{\can}$ maps $E_i$ to $E_{i+1}\otimes\Omega_Y$, and induces isomorphism $L_i\xrightarrow{\sim} L_{i+1}\otimes\Omega_Y$ for all rank-one quotients $L_i=E_i/E_{i-1}$.
\end{lem}
\begin{proof}
By hypothesis and Lem.~\ref{lem-pullback-connection}, we have a nonzero morphism $E_i\rightarrow (F_{\rel}^*E/E_i)\otimes\Omega_Y$. Let $j$ be the largest integer such that the image of $E_i$ is not contained in $(E_j/E_i)\otimes\Omega_Y$; in particular, $i\le j\le r$. Then we have an induced nonzero morphism
$$
E_i\rightarrow(E_{j+1}/E_j)\otimes\Omega_Y
$$
and it follows that $\mu(L_i)\le\mu(L_{j+1})+2g-2$. Given $\mu(L_i)=\mu(L_{i+1})+2g-2$ for all $i$, we must have $i=j$. Consequently, $\nabla_{\can}$ maps $E_i$ to $E_{i+1}\otimes\Omega_Y$, and induces a nonzero morphism $E_i\rightarrow L_{i+1}\otimes\Omega_Y$. Since $\nabla_{\can}$ also maps $E_{i-1}$ to $E_i\otimes\Omega_Y$, the above morphism $E_i\rightarrow L_{i+1}\otimes\Omega_Y$ vanishes on $E_{i-1}$. Thus we obtain a nonzero morphism $L_i\rightarrow L_{i+1}\otimes\Omega_Y$ of line bundles with the same degree, and as such, it must be an isomorphism.
\end{proof}

 Vector bundles $E$ whose pullbacks $F_{\rel}^*E$ satisfy \eqref{eq-condition} display an interesting trichotomy corresponding to the relative size of their rank $r$ and the characteristic $p$. The following lemma is necessary for our discussion of the $r<p$ case.

\begin{lem}
\label{lem-not-contain}
Let $E$ be a vector bundle over $X$, whose pullback $F_{\rel}^*E$ has Harder-Narasimhan filtration $0=E_0\subset E_1\subset\cdots\subset E_l=F_{\rel}^*E$. If the canonical connection $\nabla_{\can}$ induces isomorphisms $E_i/E_{i-1}\xrightarrow{\sim} E_{i+1}/E_i\otimes\Omega_Y$ for all $i$, then $F^*_{\rel}E'\not\subset E_{l-1}$ for every subbundle $E'\subset E$.
\end{lem}
\begin{proof}
Suppose the converse that $F^*_{\rel}E'\subset E_i$ for some $i\le l-1$, and $F^*_{\rel}E'\not\subset E_{i-1}$. Then we have a commutative square
$$
\xymatrix{F_{\rel}^*E'\ar[r]^{\text{nonzero}\quad\quad}\ar[d]_{\nabla_{\can}} & E_i/E_{i-1}\otimes\Omega_Y \ar@{=}[d]^{\nabla_{\can}} \\ F_{\rel}^*E'\otimes\Omega_X\ar[r]^{\text{zero}\quad} & E_{i+1}/E_i\otimes\Omega_Y}
$$
which gives a contradiction since the upper composition is nonzero, but the lower one is zero.
\end{proof}

\begin{prop}[$r<p$ case]
\label{prop-r<p}
Suppose $g\ge 2$ and $g-1$ is not divisible by $p$. Let $E$ be a vector bundle over $X$ of rank $r<p$ such that $F_{\rel}^*E$ satisfies \eqref{eq-condition}, then $E$ is stable.
\end{prop}
\begin{proof}
Let $0=E_0\subset E_1\subset\cdots \subset E_r=F_{\rel}^*E$ be the Harder-Narasimhan filtration of $F_{\rel}^*E$, with rank-one quotients $L_1,\cdots, L_r$. We first claim that under the hypothesis, none of the subbundles $E_i$, for $i=1,\cdots,r-1$, is a pullback along $F_{\rel}$. Indeed, since $\deg
(L_i)=\deg(L_{i+1})+2g-2$, we find
\begin{equation}
\label{eq-deg-ei}
\deg(E_i)=\sum_{j=1}^i\deg(L_j)=i\deg(L_1)-i(i-1)(g-1)
\end{equation}
In particular, for $i=r$, we obtain $\deg(L_1)=\mu(F_{\rel}^*E)+(r-1)(g-1)$. Substitute this equation back into \eqref{eq-deg-ei} and simplify, we find
$$
\deg(E_i)=i\mu(F_{\rel}^*E)+i(r-i)(g-1)=\frac{ip\deg(E)}{r}+i(r-i)(g-1)
$$
Now, suppose to the contrary of the claim that $E_i=F^*E'$ for some subbundle $E'$ of $E$. Then
\begin{equation}
\label{eq-deg-g}
\deg(E')=\frac{i\deg(E)}{r}+\frac{i(r-i)(g-1)}{p}
\end{equation}
is an integer. However, by the hypothesis, $p$ does not divide $i(r-i)(g-1)$ for all $i=1,\cdots,r-1$, and $r$ is coprime to $p$. Hence \eqref{eq-deg-g} cannot possibly be an integer, and the claim is proved.

Now, let $E'$ be any \emph{stable} proper subbundle of $E$. It suffices to show that $\mu(E')<\mu(E)$. Lem.~\ref{lem-condition-equiv} shows that $\nabla_{\can}$ induces isomorphisms $L_i\xrightarrow{\sim}L_{i+1}\otimes\Omega_Y$ for all $i$. Therefore, by Lem.~\ref{lem-not-contain}, $F_{\rel}^*E'\not\subset E_{r-1}$, and we obtain a nonzero morphism $F_{\rel}^*E'\rightarrow L_r$. Hence
\begin{equation}
\label{eq-lmin-e'}
L_{\min}(E')\le\frac{\mu_{\min}(F_{\rel}^*E')}{p}\le\frac{\deg(L_r)}{p}
\end{equation}
A computation similar to \eqref{eq-deg-ei} shows that
\begin{equation}
\label{eq-deg-lr}
p\deg(E)=\deg(F_{\rel}^*E)=r\deg(L_r)+r(r-1)(g-1)
\end{equation}
Let $r'$ be the rank of $E'$. Then the second inequality in Prop.~\ref{prop-langer-ineq}, applied to the stable bundle $E'$, shows that
$$
\mu(E')-\frac{(r'-1)(g-1)}{p}\le L_{\min}(E')\le\frac{\deg(L_r)}{p}\le\frac{\deg(E)}{r}-\frac{(r-1)(g-1)}{p}
$$
where in the second and third inequalities, we used \eqref{eq-lmin-e'} and \eqref{eq-deg-lr}, respectively. Simplify this inequality, and we obtain
$$
\mu(E')\le\mu(E)-\frac{(r-r')(g-1)}{p}<\mu(E)
$$
by the hypotheses that $g\ge 2$ and $r>r'$.
\end{proof}

\noindent
The condition that $g-1$ is not divisible by $p$ is necessary in Prop.~\ref{prop-r<p}. Indeed, one can construct counterexamples to Prop.~\ref{prop-r<p} when this condition is dropped.

\begin{rem}
\label{rem-none-hn-stable}
Suppose $g\ge 2$. Let $E$ be a vector bundle over $X$ such that $F_{\rel}^*E$ satisfies \eqref{eq-condition}. Let $0=E_0\subset E_1\subset\cdots\subset E_r=F_{\rel}^*E$ be its Harder-Narasimhan filtration. The proof of Prop.~\ref{prop-r<p} shows that if none of the $E_i$'s ($1\le i\le r-1$) comes from pulling back along $F_{\rel}$, then $E$ is stable.

Suppose now that $E_{i_j}=F_{\rel}^*G_j$ are those subbundles coming from pullbacks. Applying the argument to $E_{i_{j+1}}/E_{i_j}=F^*(G_{j+1}/G_j)$ shows that $G_{j+1}/G_j$ is stable. Furthermore, the inequality
$$
\mu(G_j/G_{j-1})=\frac{\mu(E_{i_j}/E_{i_{j-1}})}{p}>\frac{\mu(E_{i_{j+1}}/E_{i_j})}{p}=\mu(G_{j+1}/G_j)
$$
proves that $0=G_0\subset G_1\subset\cdots\subset G_j\subset\cdots\subset E$ is in fact the Harder-Narasimhan filtration of $E$, and has stable quotients.
\end{rem}

\begin{prop}[$r>p$ case]
\label{prop-r>p}
Suppose $g\ge 2$. Let $E$ be a vector bundle over $X$ of rank $r>p$ such that $F_{\rel}^*E$ satisfies \eqref{eq-condition}, then $E$ is not semistable.
\end{prop}
\begin{proof}
Again let $0=E_0\subset E_1\subset\cdots \subset E_r=F_{\rel}^*E$ be the Harder-Narasimhan filtration of $F_{\rel}^*E$, with rank-one quotients $L_1,\cdots, L_r$. Since $\mu(E_i)>\mu(F_{\rel}^*E)$ for each $i=1,\cdots,r-1$, we see that if at least one such $E_i$ is a pullback along $F_{\rel}$, then $E$ is not semistable.

Hence, we may suppose to the contrary that none of the $E_i$'s ($1\le i\le r-1$) is a pullback along $F_{\rel}$. We pick a sufficiently small open affine $\Spec(A)\subset X$, whose preimage under $F_{\rel}$ is $\Spec(B)$, where $B\cong A[x]/(x^p-a)$ for some $a\in A$, and the inclusion $A\rightarrow A[x]/(x^p-a)$ corresponds to $F_{\rel}$. We use $M_i$ to denote the free $B$-module corresponding to the vector bundle $E_i$, and $M$ for $E$. By shrinking $\Spec(A)$ if necessary, we may assume $\Omega_{B/k}\cong Bdx$ is free as well. Let $e_1,\cdots, e_r\in M\otimes_AB$ be such that for each $i$, the elements $e_1,\cdots, e_i$ form a $B$-basis for $M_i$. Write $\nabla$ for $\nabla_{\can}$, and by Lem.~\ref{lem-condition-equiv}, we may suppose
$$
\nabla(e_i)=\sum_{j=1}^{i+1}\omega_{ij}e_j
$$
where $\omega_{ij}=b_{ij}dx\in\Omega_{B/k}^1$ form a matrix of 1-forms, and $\omega_{i(i+1)}\neq 0$ for all $i$. On the other hand, let $\theta=\partial/\partial x$, so that
$$
\nabla_{(\theta^p)}(e_i)=\sum_{j=1}^{i+1}\omega_{ij}(\theta^p)e_j=\sum_{j=1}^{i+1}b_{ij}\theta^p(x)e_j=0
$$

We now evaluate the $p$-curvature of $\nabla$ on $\theta$ and $e_1$:
$$
(\nabla_{\theta})^p(e_1)-\nabla_{(\theta^p)}(e_1)=(\nabla_{\theta})^p(e_1)
$$
Note that the coefficient of $(\nabla_{\theta})^p(e_1)$ in front of $e_{p+1}$ is given by the expression $\omega_{12}(\theta)\omega_{23}(\theta)\cdots\omega_{p(p+1)}(\theta)$. Since $\omega_{i(i+1)}\neq 0$ for all $i$ and the $p$-curvature vanishes, we must have $e_{p+1}=0$. This is absurd given $r>p$.
\end{proof}

\noindent
The critical value $r=p$ is a more intricate case. However, using results of K.~Joshi, S.~Ramanan, E.~Xia, and J.-K.~Yu \cite[\S5.3]{Jos}, we can classify all semistable vector bundles of rank $r=p$ that are maximally destabilized by the Frobenius morphism.

\begin{lem}
\label{lem-pushforward-line-bundle}
Let $L$ be a line bundle over $Y$, then the pushforward $(F_{\rel})_*L$ is stable if $g\ge 2$, and semistable if $g=1$.
\end{lem}
\noindent
We point out that this lemma is neither new, nor the most general to date. It is proved by Lange and Pauly \cite[Prop.~1.2]{Lange} using the classical inequality \eqref{eq-classical-ineq} and relative duality, and later generalized by Sun \cite[Thm.~2.2]{Sun} to the case where $L$ is replaced by any stable vector bundle. We provide a simple proof using the improved inequality \eqref{eq-langer-ineq}.
\begin{proof}
Let $E$ be a rank-$r$ stable subbundle of $(F_{\rel})_*L$. Then the inclusion $E\subset (F_{\rel})_*L$ gives, by adjunction, a nonzero morphism $F_{\rel}^*E\rightarrow L$. Hence $\mu_{\min}(F_{\rel}^*E)\le\deg(L)$. The Riemann-Roch formula implies that $\deg(L)=\deg((F_{\rel})_*L)-(p-1)(g-1)$. Apply Prop.~\ref{prop-langer-ineq} to the first inequality in the following chain:
\begin{align*}
\mu(E)-\frac{(r-1)(g-1)}{p}\le& L_{\min}(E)\le\frac{\mu_{\min}(F_{\rel}^*E)}{p}\le\frac{\deg(L)}{p}\\
=&\mu((F_{\rel})_*L)-\frac{(p-1)(g-1)}{p}
\end{align*}
Therefore $\mu(E)\le\mu((F_{\rel})_*L)-(p-r)(g-1)/p$. If $g=1$, then $\mu(E)\le\mu((F_{\rel})_*L)$, and $(F_{\rel})_*L$ is semistable. If $g\ge 2$, then $\mu(E)<\mu((F_{\rel})_*L)$ for the proper subbundle $E$, and $(F_{\rel})_*L$ is stable.
\end{proof}

\begin{prop}[$r=p$ case]
\label{prop-r=p}
Suppose $g\ge 2$. Let $E$ be a vector bundle over $X$ of rank $r=p$. Then the following are equivalent:
\begin{enumerate}[(i)]
	\item $E$ is stable, and $F_{\rel}^*E$ satisfies \eqref{eq-condition};
	\item $E$ is semistable, and $F_{\rel}^*E$ satisfies \eqref{eq-condition};
	\item $E=(F_{\rel})_*L$ for some line bundle $L$.
\end{enumerate}
\end{prop}
\begin{proof}
(i) trivially implies (ii). To show that (ii) implies (iii), let $L$ denote the last rank-one quotient in the Harder-Narasimhan filtration of $F_{\rel}^*E$. Then we have a nonzero morphism $F_{\rel}^*E\rightarrow L$, which by adjunction, gives rise to a nonzero morphism $E\rightarrow (F_{\rel})_*L$. Under the assumption $r=p$, the computation \eqref{eq-deg-lr} shows that
$$
\deg(E)=\deg(L)+(p-1)(g-1)
$$
On the other hand, we obtain from the Riemann-Roch theorem that $\deg(L)=\deg((F_{\rel})_*L)-(p-1)(g-1)$. Hence
$$
\deg((F_{\rel})_*L)=\deg(L)+(p-1)(g-1)=\deg(E)
$$
and consequently $\mu(E)=\mu((F_{\rel})_*L)$. Note that $(F_{\rel})_*L$ is stable by Lem.~\ref{lem-pushforward-line-bundle}. Because $E$ is semistable and $\mu(E)=\mu((F_{\rel})_*L)$, the morphism $E\rightarrow (F_{\rel})_*L$ is necessarily surjective. On the other hand, $\rank(E)=r=p=\rank((F_{\rel})_*L)$, so this morphism is an isomorphism.

It remains to show that (iii) implies (i). Indeed, $E$ is stable by Lem.~\ref{lem-pushforward-line-bundle}. On the other hand, K.~Joshi \emph{et al.~}(cf.~\cite[\S5]{Jos}, although a detailed proof of this result is given in \cite[Lem.~2.1]{Sun}) proved that the subbundles $0=E_0\subset E_1\subset\cdots\subset E_l=F_{\rel}^*E$ defined by
\begin{enumerate}[(i)]
	\item $E_l=F^*E$ and $E_{l-1}=\ker(E_l\rightarrow L)$, where the morphism $E_l\rightarrow L$ is the counit of the adjunction, as $E_l=F_{\rel}^*(F_{\rel})_*L$.
	\item $E_i=\ker(E_{i+1}\xrightarrow{\nabla_{\can}}F_{\rel}^*E\otimes\Omega_Y\rightarrow(F_{\rel}^*E/E_{i+1})\otimes\Omega_Y)$ for $0\le i\le p-2$.
\end{enumerate}
is a pre-oper, i.e. $\nabla_{\can}(E_i)\subset E_{i+1}\otimes\Omega_Y$ for $0\le i\le l-1$, and the induced $\mathcal{O}_Y$-linear morphism $E_i/E_{i-1}\rightarrow(E_{i+1}/E_i)\otimes\Omega_Y$ is an isomorphism for $1\le i\le l-1$. In particular, the quotient $E_l/E_{l-1}$ is line bundle $L$, and inductively, each quotient $E_i/E_{i-1}\cong(E_{i+1}/E_i)\otimes\Omega_Y$ is a line bundle $L_i$, for $i=1,\cdots,l-1$. So $l=r$, the subbundles $0=E_0\subset E_1\subset\cdots\subset E_r=F_{\rel}^*E$ is the Harder-Narasimhan filtration of $F^*E$ with line bundle quotients $L_1,\cdots,L_{r-1},L_r=L$, and $\deg(L_i)=\deg(L_{i+1})+2g-2$.
\end{proof}

 Having concluded the discussion on vector bundles whose pullback satisfies \eqref{eq-condition}, we make the following definition. As before, $X$ denotes a smooth projective curve over an algebraically closed field $k$ of characteristic $p\neq 0$, with genus $g\ge 1$.

\begin{defn}
\label{defn-max-destab}
A vector bundle $E$ over $X$ is \emph{maximally Frobenius-destabilized} if $E$ is stable, the Harder-Narasimhan filtration of $F_{\rel}^*E$ consists entirely of rank-one quotients $L_1,\cdots,L_r$, and $\mu(L_i)=\mu(L_{i+1})+2g-2$ for all $i=1,\cdots,r$ (cf.~condition \eqref{eq-condition}).
\end{defn}

\noindent
When $g=1$, there is no maximally Frobenius-destabilized vector bundles for trivial reasons. Prop.~\ref{prop-r>p} shows that maximally Frobenius-destabilized bundles necessarily have rank $r\le p$. Prop.~\ref{prop-r=p} guarantees the existence of rank-$p$ maximally Frobenius-destabilized bundles over an arbitrary curve of genus $g\ge 2$. In general, we cannot expect such bundles with rank $r<p$ to exist over an \emph{arbitrary} curve. However, for most $g$, there exists a curve $X$ of genus $g$ and a maximally Frobenius-destabilized bundle $E$ over $X$ of arbitrary rank $r$ with $2\le r<p$. We construct such bundles in the following way:

\begin{lem}[S.~Mochizuki, B.~Osserman]
\label{lem-r2g2}
Suppose $p>2$. Then there exists a rank-$2$ maximally Frobenius-destabilized vector bundle $E$ over a general curve $X$ of genus $2$.
\end{lem}
\begin{proof}
Osserman's main theorem (cf.~\cite[Thm.~1.2]{Oss}, although the part we are using here is originally due to Mochizuki \cite{Moc}) states that, over a general curve of genus 2, the number of rank-$2$ semistable vector bundles with trivial determinant and whose pullback under $F$ is not semistable equals $2(p^3-p)/3>0$.

We argue that every such vector bundle $E$ is maximally Frobenius-destabilized (this is \cite[Prop.~3.3]{Jos-Xia}). Indeed, since $F_{\rel}^*E$ is not semistable and of trivial determinant, there exists a line bundle $L$ of positive degree admitting an exact sequence
$$
0\rightarrow L\rightarrow F_{\rel}^*E\rightarrow L^{-1}\rightarrow 0
$$
Furthermore, Lem.~\ref{lem-pullback-connection} gives a nonzero morphism $L\rightarrow L^{-1}\otimes\Omega_Y$, so $2\deg(L)\le\deg(\Omega_X)=2$. Thus we must have $\deg(L)=1$, and $\deg(L)=\deg(L^{-1})+2$ as desired.
\end{proof}

\begin{lem}
\label{lem-sym-power}
If a rank-$2$ vector bundle $E$ over a curve of genus $g\ge 2$ satisfies condition \eqref{eq-condition}, then so do its symmetric and exterior powers $S^t(E)$ and $\Lambda^t(E)$ for all $t\in\mathbb{N}$.
\end{lem}
\noindent
We prove the lemma for $S^t(E)$. The proof for the exterior power case is completely analogous.
\begin{proof}
Indeed, a rank-$2$ vector bundle $E$ satisfies \eqref{eq-condition} if and only if there exist line bundles $L_1$, $L_2$, and an exact sequence
$$
0\rightarrow L_1\rightarrow E\rightarrow L_2\rightarrow 0
$$
such that $\deg(L_1)=\deg(L_2)+2g-2$. Now $S^t(E)$ has a filtration $0=E_0\subset E_1\subset\cdots\subset E_{t+1}=S^t(E)$ with rank-one quotients
$$
E_{i+1}/E_i\cong S^{t-i}(L_1)\otimes S^{i}(L_2)\cong L_1^{t-i}\otimes L_2^i
$$
so $\deg(E_{i+1}/E_i)=t\deg(L_1)-i(2g-2)$. Thus this filtration is the Harder-Narasimhan filtration of $S^t(E)$, and its successive quotients satisfy $\deg(E_i/E_{i-1})=\deg(E_{i+1}/E_i)+2g-2$.
\end{proof}

\begin{prop}
\label{prop-any-g-any-r}
Suppose natural number $p$, $g$, and $r$ are all at least $2$, and either
\begin{enumerate}[(i)]
	\item $r=p$, or
	\item $r<p$ and $p$ does not divide $g-1$,
\end{enumerate}
Then there exists a rank-$r$ maximally Frobenius-destabilized vector bundle $E$ over some genus-$g$ curve $X$.
\end{prop}
\begin{proof}
Existence in case (i) follows directly from \eqref{prop-r=p}, and holds over an \emph{arbitrary} curve of genus $g$.

For case (ii), we first consider a general curve $Z$ of genus $2$. By Lem.~\ref{lem-r2g2}, there exists a rank-2 vector bundle $G$ over $Z$ which is maximally Frobenius-destabilized.

Fix $n=g-1$. Then $n$ is not divisible by $p$, so finite \'{e}tale coverings of $Z$ of degree $n$ are parametrized by the $n$-torsion elements of the Picard group $\Pic(Y)[n]\cong(\ZZ/n\ZZ)^{2g}$. In particular, there exists such a covering $f:X\rightarrow Z$ of degree $n$. Riemann-Hurwitz formula shows that the genus $g_X$ of $X$ satisfies $g_X=n+1=g$. Furthermore, $f^*\Omega_Z=\Omega_X$, so the exact sequence (cf.~Rem.~\ref{rem-max-destab})
$$
0\rightarrow L\otimes\Omega_{Z^{(-1)}}\rightarrow F_{\rel}^*G\rightarrow L\rightarrow 0
$$
pulls back to an exact sequence under the flat morphism $f$:
\begin{equation}
\label{eq-f*G}
0\rightarrow f^*L\otimes\Omega_Y\rightarrow F_{\rel}^*f^*G\rightarrow f^*L\rightarrow 0
\end{equation}
In particular, $f^*G$ is a rank-2 vector bundle over $X$ whose pullback along $F_{\rel}$ satisfies condition \eqref{eq-condition}. Taking $E=S^{r-1}(f^*G)$ and applying Lem.~\ref{lem-sym-power} and Prop.~\ref{prop-r<p} completes the construction of the rank-$r$ vector bundle $E$.
\end{proof}

 We also note the following

\begin{cor}
\label{cor-sym-pushforward-lb}
Suppose $p>2$, $g\ge 2$ and $g-1$ is not divisible by $p$. Let $E$ be a rank-$2$ vector bundle over a genus-$g$ curve such that $F_{\rel}^*E$ satisfies \eqref{eq-condition}. Then $S^{p-1}(E)=(F_{\rel})_*L$ for some line bundle $L$.
\end{cor}
\begin{proof}
$F_{\rel}^*S^{p-1}(E)$ also satisfies \eqref{eq-condition}, so we know from Prop.~\ref{prop-r=p} that it suffices to prove the stability of $S^{p-1}(E)$. Let $0=E_0\subset E_1\subset\cdots\subset E_p=F_{\rel}^*S^{p-1}(E)$ be the Harder-Narasimhan filtration of $F_{\rel}^*S^{p-1}(E)$. We only need to show that none of the $E_i$'s ($1\le i\le p-1$) comes from pulling back along $F_{\rel}$ (cf.~Rem.~\ref{rem-none-hn-stable}). If some $E_i=F_{\rel}^*E'$ for some subbundle $E'$ of $S^{p-1}(E)$, then the computation \eqref{eq-deg-g} shows that
\begin{equation}
\label{eq-deg-g-r}
\deg(E')=\frac{i\deg(S^{p-1}(E))}{p}+\frac{i(p-i)(g-1)}{p}
\end{equation}
On the other hand, it follows from a Chern class computation that
$$
\deg(S^{p-1}(E))=\frac{p(p-1)}{2}\deg(E)
$$
Therefore, under the hypothesis that $g-1$ is not divisible by $p$, the expression in \eqref{eq-deg-g-r} cannot be an integer for $1\le i\le p-1$.
\end{proof}

\noindent
and a partial improvement of the bound in Prop.~\ref{prop-langer-ineq} using a theorem of X.~Sun:

\begin{prop}
\label{prop-langer-ineq-r>p}
Suppose $g\ge 2$. Let $E$ be a semistable vector bundle over a genus-$g$ curve $X$. Then
$$
L_{\max}(E)-\mu(E)<g-1\quad\text{and}\quad\mu(E)-L_{\min}(E)<g-1
$$
\end{prop}
\noindent
Let $r$ denote the rank of $E$. When $r>p$, these inequalities are sharper than Prop.~\ref{prop-langer-ineq}.
\begin{proof}
We first prove the second inequality. Using Thm.~\ref{thm-fdHN}, we may fix some sufficiently large $k$ such that $L_{\min}(E)=\mu_{\min}((F_{\rel}^k)^*E)/p^k$. Consider the last quotient $Q_l$ in the Harder-Narasimhan filtration of $(F_{\rel}^k)^*E$, and let $Q$ be the last quotient in the Jordan-H\"{o}lder filtration of $Q_l$. Then
$$
\mu(Q)=\mu(Q_l)=\mu_{\min}((F_{\rel}^k)^*E)
$$
and there is a surjection $(F_{\rel}^k)^*E\rightarrow Q$. By adjunction, we obtain a nonzero morphism $E\rightarrow(F_{\rel}^k)_*Q$. Sun's theorem \cite[Thm.~2.2]{Sun} shows $(F_{\rel}^k)_*Q$ is stable. Hence $\mu(E)\le\mu((F_{\rel}^k)_*Q)$. Let $q$ denote the rank of $Q$. Expressing $\mu((F_{\rel}^k)_*Q)$ in terms of $\mu(Q)$ by the Riemann-Roch formula, we compute
\begin{equation}
\label{eq-better-ineq}
\mu(E)\le\mu((F_{\rel}^k)_*Q)=\frac{\mu(Q)}{p^k}+\left(1-\frac{1}{p^k}\right)(g-1)=L_{\min}(E)+\left(1-\frac{1}{p^k}\right)(g-1)
\end{equation}
and the desired inequality follows.

For the first inequality, consider the maximal destabilizing subsheaf $E_1$ of $(F_{\rel}^k)^*E$, and let $S$ be the first nontrivial subbundle in the Jordan-H\"{o}lder filtration of $E_1$. Then there is a surjection $(F_{\rel}^k)^*E^{\vee}\rightarrow S^{\vee}$, where $S^{\vee}$ is stable with $\mu(S^{\vee})=-\mu(S)=-\mu_{\max}((F_{\rel}^k)^*E)$. The above computation \eqref{eq-better-ineq} then shows that $\mu(E^{\vee})=-\mu(E)\le -L_{\max}(E)+(g-1)$, which gives the first inequality.
\end{proof}

\section{The characteristic-$3$ case}
In this section, $X$ will denote a smooth projective curve over an algebraically closed field $k$ of characteristic $p=3$, with genus $g\ge 2$. We will again denote $X^{(-1)}$ by $Y$, thus writing the relative Frobenius morphism $F_{\rel}:Y\rightarrow X$.

 We try to turn Cor.~\ref{cor-sym-pushforward-lb} around by taking a line bundle $L$, and asking whether $(F_{\rel})_*L$ can be expressed as the symmetric power of a rank-2 vector bundle. The problem can be reduced to vector bundles of trivial determinant.

\begin{lem}
\label{lem-max-fd-trivial-det}
Suppose $E$ is a rank-$2$ vector bundle, and $F_{\rel}^*E$ satisfies condition \eqref{eq-condition}, i.e. there exists a line bundle $L$ over $Y$ and an exact sequence
\begin{equation}
\label{eq-rank-two-condition}
0\rightarrow L\otimes\Omega_Y\rightarrow F_{\rel}^*E\rightarrow L\rightarrow 0
\end{equation}
Then after twisting by a line bundle, $E$ has trivial determinant.
\end{lem}
\noindent
The lemma actually holds in all positive characteristic $p\neq 2$.
\begin{proof}
Pulling back further by the isomorphism $X\xrightarrow{\sim} Y$ of abstract schemes, we see that \eqref{eq-rank-two-condition} gives rise to an exact sequence
$$
0\rightarrow L'\otimes\Omega_X\rightarrow F_{\abs}^*E\rightarrow L'\rightarrow 0
$$
over $X$, for some line bundle $L'$. Let $\Theta$ be a theta characteristic over $X$ (which exists for $p\neq 2$). Then $\Theta^2\cong\Omega_X$. On the other hand, the above exact sequence shows that
$$
3c_1(E)=c_1(F_{\abs}^*E)=c_1(L'\otimes\Omega_X)+c_1(L')=2c_1(L'\otimes\Theta)
$$
Let $M=\det(E)\otimes(L'\otimes\Theta)^{-1}$. Then
$$
c_1(E\otimes M)=c_1(E)+2c_1(M)=3c_1(E)-2c_1(L'\otimes\Theta)=0
$$
and therefore $\det(E\otimes M)\cong\mathcal{O}_X$.
\end{proof}

\begin{prop}
\label{prop-p=3-easy-direction}
Suppose $E$ is a rank-$2$ vector bundle over $X$. Then the following are equivalent:
\begin{enumerate}[(i)]
	\item $E$ is maximally Frobenius-destabilized;
	\item $S^2(E)$ is maximally Frobenius-destabilized;
	\item $S^2(E)=(F_{\rel})_*L$ for some line bundle $L$;
\end{enumerate}
Furthermore, if $E$ is assumed to have trivial determinant, then we can replace $L$ by $T_Y$ in (iii).
\end{prop}
\begin{proof}
The equivalence of (ii) and (iii) is already proved in Prop.~\ref{prop-r=p}.

For (i)$\implies$(ii), note that by Lem.~\ref{lem-sym-power} and the equivalence (i)$\iff$(ii) in Prop.~\ref{prop-r=p}, we only need to show that $S^2(E)$ is semistable. By twisting $E$ with a line bundle, we may assume $\deg(E)=0$. On the other hand, there is an exact sequence
$$
0\rightarrow N\rightarrow E^{\otimes 2}\rightarrow S^2(E)\rightarrow 0
$$
where $N$ is a line bundle. Since $E$ has degree zero by assumption, all three vector bundles in this sequence have vanishing slope. The theorem of S.~Ilangovan, V.~B.~Mehta, and A.~J.~Parameswaran \cite{Ila} shows that the tensor product of semistable vector bundles whose ranks add up to less than $p+2$ is again semistable. Therefore, $E^{\otimes 2}$ is semistable, so all quotients of $E^{\otimes 2}$ are of nonnegative degree. Since all quotients of $S^2(E)$ are quotients of $E^{\otimes 2}$, the same holds for $S^2(E)$.

For (ii)$\implies$(i), to prove the stability of $E$, let $M$ be any quotient line bundle of $E$. Therefore, we have a surjection $S^2(E)\rightarrow M^2$. But $S^2(E)$ is stable by assumption, so $\mu(S^2E)<2\deg(M)$. On the other hand, a Chern class computation shows that $\deg(S^2(E))=3\deg(E)$. Thus $2\mu(E)=\mu(S^2(E))<2\deg(M)$, and $E$ is stable. Again, because $S^2(F_{\rel}^*E)\cong F_{\rel}^*S^2(E)$ is not semistable, the Ilangovan-Mehta-Parameswaran theorem shows that $F_{\rel}^*E$ is not semistable. Thus $F_{\rel}^*E$ fits into an exact sequence
\begin{equation}
\label{eq-seq-pullback-rank-2}
0\rightarrow N\rightarrow F_{\rel}^*E\rightarrow R\rightarrow 0
\end{equation}
where $\deg(N)>\mu(F_{\rel}^*E)>\deg(R)$. Therefore, $S^2(F_{\rel}^*E)$ admits a filtration with successive quotients $N^{\otimes 2}$, $N\otimes R$, and $R^2$. The fact that they are stable, and $\deg(N^2)>\deg(N\otimes R)>\deg(R^2)$ shows that this is the Harder-Narasimhan filtration of $S^2(F_{\rel}^*E)$. On the other hand, $S^2(F_{\rel}^*E)$ satisfies condition \eqref{eq-condition}, so $N\cong R\otimes\Omega_Y$. Looking back at the exact sequence \eqref{eq-seq-pullback-rank-2}, we see that $E$ is maximally Frobenius-destabilized.

For the second claim, note that the assumption $\det(E)\cong\mathcal{O}_X$ implies $N\cong R^{-1}$ in \eqref{eq-seq-pullback-rank-2}. Hence $R^2\cong T_Y$. We also observe (cf.~proof of ``(iii)$\implies$(i)" in Prop.~\ref{prop-r=p}) that the Harder-Narasimhan filtration of $F_{\rel}^*(F_{\rel})_*L$ has successive quotients $L\otimes\Omega_Y^2$, $L\otimes\Omega_Y$, and $L$. It follows from the uniqueness of Harder-Narasimhan filtration that if $(F_{\rel})_*L\cong S^2(E)$, then $L\cong R^2\cong T_Y$.
\end{proof}

Restricting ourselves to the trivial-determinant case, we seek to express $(F_{\rel})_*T_Y$ as a second symmetric power. We do so using the following rather general machinery:

\begin{lem}
\label{lem-nondeg-quad}
There is a bijection of sets:
\begin{align*}
\Phi:&
\left( \begin{array}{c}
\text{isomorphism classes of rank-$2$} \\
\text{vector bundles $E$ over $X$} \end{array} \right)\Big/\text{twisting by line bundles}\xrightarrow{\sim}\\
&\left( \begin{array}{c}
\text{pairs $(F,q)$ where $F$ is a rank-$3$ vector bundle with trivial} \\
\text{determinant and $q\in H^0(X,S^2(F))$ is a nondegenerate quadric} \end{array}\right)\Big/\text{equivalences}
\end{align*}
where the underlying rank-$3$ bundle of $\Phi(E)$ is $S^2(E)\otimes\det(E)^{-1}$.
\end{lem}
\noindent
We do not go into details in defining equivalences of pairs $(F,q)$, although they will be clear from the proof. For our application, we will be looking at a specific vector bundle $F$, for which the choice of $q$ is unique up to scaling.
\begin{proof}
A pair $(F,q)$ of rank-$3$ vector bundle $F$ with trivial determinant and $q\in H^0(X,S^2(F))$ a nondegenerate quadric corresponds to an $\mathrm{SO}_3$-principal bundle, via the standard representation of $\mathrm{SO}_3$. On the other hand, $\mathrm{SO}_3$ equipped with the standard representation is equivalent to $\mathrm{PGL}_2$ equipped with the adjoint representation $\mathfrak{g}$. Consider the exact sequence of groups
$$
0\rightarrow\mathbb{G}_m\rightarrow \mathrm{GL}_2\xrightarrow{p} \mathrm{PGL}_2\rightarrow 0
$$
Let $V$ be the standard representation of $\mathrm{GL}_2$. Then the precomposition of $\mathfrak{g}$ by $p$ is isomorphic to the representation $S^2(V)\otimes\det(V)^{-1}$ of $\mathrm{GL}_2$. Taking cohomology, we obtain
$$
H^1_{\text{\'{e}t}}(X,\mathbb{G}_m)\rightarrow H^1_{\text{\'{e}t}}(X,\mathrm{GL}_2)\rightarrow H^1_{\text{\'{e}t}}(X,\mathrm{PGL}_2)\rightarrow H^2_{\text{\'{e}t}}(X,\mathbb{G}_m)
$$
where $H^2_{\text{\'{e}t}}(X,\mathbb{G}_m)$ vanishes by Tsen's theorem. The bijection $\Phi$ then follows from the isomorphism $H^1_{\text{\'{e}t}}(X,\mathbb{G}_m)\cong\Pic(X)$.
\end{proof}

\label{relative-duality} For later computations in Prop.~\ref{prop-nondeg-quad}, we first work out an explicit description of relative duality for the morphism $F_{\rel}$. This material is standard. Around any point $x\in X$, pick an affine neighborhood $\Spec(A)\subset X$, whose preimage under $F_{\rel}$ is given by $\Spec(B)\subset Y$, with $B\cong A[x]/(x^3-a)$ for some $a\in A$. The morphism $F_{\rel}$ corresponds to the embedding $A\rightarrow B\cong A[x]/(x^3-a)$. By further shrinking, we may assume $\Omega_{A/k}\cong Ada$ and $\Omega_{B/k}\cong Bdx$ are both free; we abbreviate $\Omega_{A/k}$ and $\Omega_{B/k}$ by $\Omega_A$ and $\Omega_B$. There is a trace map (see, for example, \cite[Tag 0ADY]{Stacks}):
\begin{equation}
\label{eq-stacks-trace}
\mathrm{tr}:(F_{\rel})_*\Omega_Y\rightarrow\Omega_X
\end{equation}
defined locally by a map $\mathrm{tr}_{x,d}:(\Omega_B)_A\rightarrow\Omega_A$ of $A$-modules:
$$
\mathrm{tr}_{x,d}(dx)=0,\quad\mathrm{tr}_{x,d}(xdx)=0,\quad\text{and}\quad\mathrm{tr}_{x,d}(x^2dx)=da
$$
By the adjunction $((F_{\rel})_*, F_{\rel}^!)$, we have a nonzero morphism $\Omega_Y\rightarrow F_{\rel}^!\Omega_X$. Locally, this morphism is defined by a map $\Omega_B\rightarrow\Hom_A(B,\Omega_A)$, sending $dx$ to the $A$-linear map:
$$
1\mapsto 0,\quad x\mapsto 0,\quad\text{and}\quad x^2\mapsto da
$$
We see from the local description that this morphism is injective and surjective. In summary,

\begin{lem}
\label{lem-shriek-isom}
There is an isomorphism $\Omega_Y\rightarrow F_{\rel}^!\Omega_X$ which, by adjunction of $((F_{\rel})_*, F_{\rel}^!)$, gives the trace map \eqref{eq-stacks-trace}.\qed
\end{lem}

We also note that the isomorphism $F_{\rel}^*\Omega_X\xrightarrow{\sim}\Omega_Y^3$, which is locally given by a map $\Omega_A\otimes_AB\rightarrow\Omega_B^3$, sends $da\otimes 1$ to $(dx)^3$. The following lemma proves itself.

\begin{lem}
\label{lem-chain-isom}
For every line bundle $L$ over $Y$, there is a chain of isomorphisms
\begin{align*}
(F_{\rel})_*L\cong& (F_{\rel})_*\mathcal{H}\mathrm{om}(L^{-1}\otimes\Omega_Y,\Omega_Y)\cong (F_{\rel})_*\mathcal{H}\mathrm{om}(L^{-1}\otimes\Omega_Y,F_{\rel}^!\Omega_X)\\
\cong&\mathcal{H}\mathrm{om}((F_{\rel})_*(L^{-1}\otimes\Omega_Y),\Omega_X)
\cong\mathcal{H}\mathrm{om}((F_{\rel})_*(L^{-1}\otimes T_Y^2\otimes\Omega_Y^3), \Omega_X)\\
\cong&\mathcal{H}\mathrm{om}((F_{\rel})_*(L^{-1}\otimes T_Y^2)\otimes\Omega_X,\Omega_X)\cong ((F_{\rel})_*(L^{-1}\otimes T_Y^2))^{\vee}
\end{align*}
where Lem.~\ref{lem-shriek-isom} is used in the second isomorphism, the adjunction $((F_{\rel})_*, F_{\rel}^!)$ in the third, and the projection formula in the fifth.\qed
\end{lem}

\noindent
In particular, $(F_{\rel})_*T_Y$ is self-dual and $(F_{\rel})_*\mathcal{O}_Y$ is dual to $(F_{\rel})_*(T_Y^2)$.

 For notational simplicity, we will now use $F$ to denote $(F_{\rel})_*T_Y$. Observe that there is a canonical surjection
\begin{equation}
\label{eq-canonical-surj}
S^2(F)\rightarrow(F_{\rel})_*(T_Y^2)\rightarrow 0
\end{equation}
Since $F$ is self-dual (cf.~\S\ref{relative-duality}), it has trivial determinant. Given any vector bundle $E$ over $X$, the nondegenerate pairing on $E^{\otimes 2}$ and $(E^{\vee})^{\otimes 2}$ induces, using $2<3$, a nondegenerate pairing on $S^2(E)$ and $S^2(E^{\vee})$. Thus, the dual of \eqref{eq-canonical-surj} gives a canonical injection
\begin{equation}
\label{eq-canonical-inj}
0\rightarrow (F_{\rel})_*\mathcal{O}_Y\rightarrow S^2(F^{\vee})\cong S^2(F)
\end{equation}
In fact, the above two morphisms combine into an exact sequence:

\begin{prop}
There is a canonically defined exact sequence
\begin{equation}
\label{eq-canonical-seq}
0\rightarrow(F_{\rel})_*\mathcal{O}_Y\rightarrow S^2(F)\rightarrow(F_{\rel})_*(T_Y^2)\rightarrow 0
\end{equation}
\end{prop}
\begin{proof}
Indeed, we have a composition $(F_{\rel})_*\mathcal{O}_Y\rightarrow S^2(F)\rightarrow(F_{\rel})_*(T_Y^2)$ from \eqref{eq-canonical-surj} and \eqref{eq-canonical-inj}. This is a zero morphism, since
$$
\deg((F_{\rel})_*\mathcal{O}_Y)=2g-2>-(2g-2)=\deg((F_{\rel})_*(T_Y^2))
$$
by a computation using the Riemann-Roch formula, and the two vector bundles are stable by Lem.~\ref{lem-pushforward-line-bundle}. Therefore the injective morphism $(F_{\rel})_*\mathcal{O}_Y\rightarrow S^2(F)$ factors through the kernel of $S^2(F)\rightarrow(F_{\rel})_*(T_Y^2)$. By rank considerations, $(F_{\rel})_*\mathcal{O}_Y$ is identified with this kernel.
\end{proof}

\begin{cor}
\label{cor-set-atmost-one}
$S^2(F)$ has a unique global section up to scaling, and it is given by the image of $1\in H^0(X,(F_{\rel})_*\mathcal{O}_Y)$ under the canonical injection \eqref{eq-canonical-inj}.
\end{cor}
\begin{proof}
\eqref{eq-canonical-seq} gives rise to an exact sequence
$$
0\rightarrow H^0(X, (F_{\rel})_*\mathcal{O}_Y)\rightarrow H^0(X,S^2(F))\rightarrow H^0(X,(F_{\rel})_*(T_Y^2))
$$
Since $(F_{\rel})_*(T_Y^2)$ is stable of negative degree $-(2g-2)$, there is no nonzero morphism $\mathcal{O}_X\rightarrow (F_{\rel})_*(T_Y^2)$. Therefore $H^0(X,(F_{\rel})_*(T_Y^2))=0$, so $H^0(X,S^2(F))=H^0(X, (F_{\rel})_*\mathcal{O}_Y)$ has a unique section up to scaling.
\end{proof}

\begin{prop}
\label{prop-nondeg-quad}
The image of $1\in H^0(X,(F_{\rel})_*\mathcal{O}_Y)$ under the canonical injection \eqref{eq-canonical-inj} is a nondegenerate quadric on every fiber of $S^2(F)$.
\end{prop}
\begin{proof}
We compute the image of $1\in H^0(X,(F_{\rel})_*\mathcal{O}_Y)$ in $H^0(X,S^2(F))$ explicitly. In order to do so, we have to understand the isomorphism $(F_{\rel})_*\mathcal{O}_Y\cong ((F_{\rel})_*(T_Y^2))^{\vee}$ given by relative duality, find out the image of $1$, and then dualize the canonical surjection $S^2(F)\rightarrow (F_{\rel})_*(T_Y^2)$.

Around any point $x\in X$, pick a sufficiently small affine neighborhood, and use the notations in \S\ref{relative-duality} for the relevant local expressions. The chain of isomorphisms in Lem.~\ref{lem-chain-isom} can be expressed locally as
$$
\xymatrix@C=-20pt{ & \Hom_B(\Omega_B, \Hom_A(B,\Omega_A))_A \ar@{=}[rr] & & \Hom_A((\Omega_B)_A,\Omega_A) \ar@{=}[dr]_{\text{contraction}\quad} & \\
\Hom_B(\Omega_B,\Omega_B)_A\ar@{=}[ur]^{F_{\rel}^!\Omega_X\cong\Omega_Y\quad} \ar[rrru]_{\quad\mathrm{identity}\mapsto\mathrm{tr}_{x,d}} & & & & \Hom_A((T_B^2\otimes\Omega_B^3)_A,\Omega_A)\ar@{=}[d]_{\text{projection formula}} \\
B_A\ar@{=}[u]_{\text{contraction}} & & ((T_B^2)_A)^{\vee}\ar@{=}[ll]_{?} & & \Hom_A((T_B^2)_A\otimes\Omega_A,\Omega_A)\ar@{=}[ll]_{\text{contraction}\quad\quad\quad}}
$$
where $1\in B_A$ is sent to the identity morphism in $\Hom_B(\Omega_B,\Omega_B)_A$, and consequently the map $\mathrm{tr}_{x,d}$ in $\Hom_A((\Omega_B)_A,\Omega_A)$, which is then the morphism
$$
\left(\frac{\partial}{\partial x}\right)^{\otimes 2}\otimes da\mapsto 0,\quad x\left(\frac{\partial}{\partial x}\right)^{\otimes 2}\otimes da\mapsto 0,\quad\text{and}\quad x^2\left(\frac{\partial}{\partial x}\right)^{\otimes 2}\otimes da\mapsto da
$$
in $\Hom_A((T_B^2)_A\otimes\Omega_A,\Omega_A)$ after contraction and applying the projection formula. This is then finally the element
$$
\left(\frac{\partial}{\partial x}\right)^{\otimes 2}\mapsto 0,\quad x\left(\frac{\partial}{\partial x}\right)^{\otimes 2}\mapsto 0,\quad\text{and}\quad x^2\left(\frac{\partial}{\partial x}\right)^{\otimes 2}\mapsto 1
$$
in $((T_B^2)_A)^{\vee}$. Let $X_1=\partial/\partial x$, $X_2=x(\partial/\partial x)$, and $X_3=x^2(\partial/\partial x)$ be an $A$-basis for $(T_B)_A$, and $\{X_iX_j\}$ a basis for $S^2((T_B)_A)$. The composition $B_A\xrightarrow{\sim}((T_B^2)_A)^{\vee}\rightarrow S^2((T_B)_A)^{\vee}$ sends $1\in B_A$ to the sum $(X_1X_3)^{\vee}+(X_2^2)^{\vee}$, where $(X_1X_3)^{\vee}$ denotes the dual basis for $X_1X_3$, and similarly for $(X_2^2)^{\vee}$. Since the isomorphism for symmetric powers$S^2((T_B)_A)^{\vee}\cong S^2(((T_B)_A)^{\vee})$ is induced from that for the tensor powers, the element in $S^2(((T_B)_A)^{\vee})$ corresponding to $(X_1X_3)^{\vee}+(X_2^2)^{\vee}$ is $X_1^{\vee}X_3^{\vee}+(X_2^{\vee})^2$. This expression is a nondegenerate quadric. Since nondegeneracy is preserved under linear equivalences, the corresponding expression in $S^2((T_B)_A)$ via self-duality of $(T_B)_A$ is still nondegenerate. The proof is complete, as the choice of $x\in X$ is arbitrary.
\end{proof}

\begin{cor}
\label{cor-trivial-det}
There exists a rank-$2$ vector bundle $E$ over $X$ with trivial determinant, such that $S^2(E)\cong F$. Furthermore, the choice of such a vector bundle $E$ is unique up to twisting by a line bundle in $\Pic(X)[2]$.
\end{cor}
\begin{proof}
Prop.~\ref{prop-nondeg-quad} and Cor.~\ref{cor-set-atmost-one} together show that there exists a unique nondegenerate quadric $q\in H^0(X,S^2(F))$ up to scaling. Therefore, Lem.~\ref{lem-nondeg-quad} shows that all rank-$2$ vector bundles $E$ with $S^2(E)\otimes\det(E)^{-1}\cong F$ occur, by uniqueness of $q$, as preimage of the pair $(F,q)$ under $\Phi$. Thus they all differ by line bundle twists. Given such a vector bundle $E$, since $S^2(E)\cong F\otimes\det(E)$ is maximally Frobenius-destabilized, the same holds for $E$ itself (Prop.~\ref{prop-p=3-easy-direction}). It follows from Lem.~\ref{lem-max-fd-trivial-det} that after twisting by a line bundle, $E$ has trivial determinant, and thus $S^2(E)\cong F$. Finally, note that given two rank-$2$ vector bundles $E$, $E'$ with trivial determinant such that $E'=E\otimes L$ for some line bundle $L$, then $L\in\Pic(X)[2]$ because $\det(E')\cong\det(E)\otimes L^2$.
\end{proof}

 We summarize the results of this section.

\begin{thm}
\label{thm-trivial-det}
In characteristic $3$, maximally Frobenius-destabilized vector bundles of rank-$2$ with trivial determinant exist over an arbitrary smooth projective curve of genus $g\ge 2$, and are unique up to twisting by an arbitrary line bundle in $\Pic(X)[2]$. Furthermore, such vector bundles are precisely the rank-$2$ vector bundles $E$ satisfying $S^2(E)=(F_{\rel})_*T_Y$.
\end{thm}
\noindent
In particular, there are $\#\Pic(X)[2]$ number of such vector bundles.
\begin{proof}
We already know from Prop.~\ref{prop-p=3-easy-direction} that rank-$2$ maximally Frobenius-destabilized vector bundles with trivial determinant are precisely the rank-$2$ vector bundles $E$ satisfying $S^2(E)\cong(F_{\rel})_*T_Y$. On the other hand, by Cor.~\ref{cor-trivial-det}, bundles with this property exist and are unique up to twisting by any line bundle in $\Pic(X)[2]$.
\end{proof}

\noindent
Even more generally,

\begin{thm}
In characteristic $3$, maximally Frobenius-destabilized vector bundles of rank-$2$ exist over an arbitrary smooth projective curve of genus $g\ge 2$, and are unique up to twisting by line bundles. Furthermore, such vector bundles are precisely the rank-$2$ vector bundles $E$ satisfying $S^2(E)\cong (F_{\rel})_*L$ for some line bundle $L$ over $Y$.
\end{thm}
\begin{proof}
The second claim again follows from Prop.~\ref{prop-p=3-easy-direction}. Up to twisting by a line bundle, every rank-$2$ maximally Frobenius-destabilized vector bundle $E$ has trivial determinant, by Lem.~\ref{lem-max-fd-trivial-det}. The result then reduces to thm.~\ref{thm-trivial-det}.
\end{proof}

\begin{rem}
In \cite[Thm.~1.2]{Oss}, B.~Osserman proves that over an arbitrary genus-$2$ curve in characteristic $3$, there are exactly $\#\Pic(X)[2]$ number of rank-$2$ Frobenius-destabilized vector bundles. Since over genus-2 curves, Frobenius-destabilized bundles are precisely the maximally Frobenius-destabilized ones (cf.~\cite[Prop.~3.3]{Jos}, or look in the proof of our Lem.~\ref{lem-r2g2}), Thm.~\ref{thm-trivial-det} specializes to Osserman's counting formula in the genus-$2$ case.
\end{rem}

\section{An application: ample vector bundles}
 We first give a new criterion for ample vector bundles, which is valid over an arbitrary smooth projective variety over an algebraically closed field $k$ (of arbitrary characteristic). C.~Barton \cite{Bar} proved a numerical criterion for ample vector bundles in terms of their pullbacks. His result will form the basis of ours, so we briefly review it. Given a projective variety $X$ over an algebraically closed field $k$, we let $N_1(X)$ denote the group of integral $1$-cycles of $X$ modulo numerical equivalence, and let $A_1(X)=N_1(X)\otimes_{\ZZ}\RR$. It is a finite-dimensional vector space (cf.~\cite[\S IV, Prop.~4]{Kle}). Thus we may fix a basis $z_1,\cdots,z_q$ for $A_1(X)$, and let $\|\cdot\|$ be the norm on $A_1(X)$ for which the basis $z_1,\cdots,z_q$ is orthonormal.

\begin{thm}[C.~Barton]
\label{thm-barton}
A vector bundle $E$ over $X$ is ample if and only if there exists a real number $\varepsilon$ such that for any smooth projective curve $Y$, any finite morphism $f:Y\rightarrow X$, and any line bundle quotient $f^*E\rightarrow L$, there holds $\deg(L)/\deg(f)\ge\varepsilon\|f(Y)\|$.
\end{thm}
\begin{proof}
This is \cite[Thm.~2.1]{Bar}. The main ingredient of the proof is a criterion of amplitude due to S.~Kleiman \cite[IV-2, Prop.~2]{Kle}.
\end{proof}

\noindent
We also need a result about vector bundles over curves proved by S.-W.~Zhang in the arithmetic setting (though it is possibly older), and we choose the more geometric presentation in \cite{dJS}:

\begin{thm}[S.-W.~Zhang]
\label{thm-zhang}
Given a vector bundle $E$ over a smooth projective curve $C$ and any real number $\varepsilon>0$, there exists a smooth projective curve $C'$, a finite morphism $f:C'\rightarrow C$, and a line bundle quotient $f^*E\rightarrow L$ such that $\deg(L)/\deg(f)<\mu(E)+\varepsilon$.
\end{thm}
\begin{proof}
See \cite[Thm.~5.7]{dJS}. Although the authors prove that there exists a finite, flat cover of $C$, their proof actually shows that $C'$ can be chosen to be smooth.
\end{proof}

Our criterion generalizes \cite[Thm.~2.3]{Bre} and \cite[Thm.~1.1]{Bis} to arbitrary dimensions, and is essentially a consequence of \cite[IV-2, Prop.~2]{Kle}:

\begin{thm}
\label{thm-ample-criterion-text}
Let $X$ be a smooth projective variety over an algebraically closed field $k$, and $E$ be a vector bundle over $X$. Then $E$ is ample if and only if there exists some real number $\varepsilon>0$, such that for any integral, closed curve $C$ in $X$, there holds $L_{\min}(E|_C)\ge\varepsilon\|C\|$.
\end{thm}

\noindent
As before, the norm $\|\cdot\|$ on $A_1(X)$ is fixed with respect to a chosen basis. If $C$ is a non-smooth curve, the notation $L_{\min}(E|_C)$ is used to denote $L_{\min}(\nu^*E)$ where $\nu:\tilde{C}\rightarrow C$ is the normalization map. The proof in characteristic zero is a simplified version of the proof in positive characteristic, so we omit the former.

\begin{proof}[Proof in characteristic $p\neq 0$]
First suppose that $E$ is ample. Thm.~\ref{thm-barton} provides some $\varepsilon'>0$ such that for any smooth projective curve $Y$, any finite morphism $f:Y\rightarrow X$, and any line bundle quotient $f^*E\rightarrow L$, Barton's inequality $\deg(L)/\deg(f)\ge\varepsilon'\|f(Y)\|$ holds.

Now consider any integral, closed curve $C$ in $X$, and let $\nu:\tilde{C}\rightarrow C$ be the normalization map (with $\tilde{C}\cong C$ if $C$ is already smooth). Let $\tilde{E}=\nu^*(E|_C)$. By Thm.~\ref{thm-fdHN}, we may fix a sufficiently large $k$, such that
$$
L_{\min}(E|_C)=L_{\min}(\tilde{E})=\frac{\mu_{\min}((F_{\rel}^k)^*\tilde{E})}{p^k}
$$
Let $Q$ be the last quotient in the Harder-Narasimhan filtration of $(F_{\rel}^k)^*\tilde{E}$. Then $\mu_{\min}((F_{\rel}^k)^*\tilde{E})=\mu(Q)$. Furthermore, Thm.~\ref{thm-zhang} allows us to find a smooth projective curve $Y$, a finite morphism $g:Y\rightarrow\tilde{C}^{(-k)}$, and a line bundle quotient $g^*Q\rightarrow L$ such that
$$
\frac{\deg(L)}{\deg(g)}<\mu(Q)+\frac{p^k\varepsilon'}{2}\cdot\|C\|
$$
Divide by $p^k$, and let $f=\nu\circ F_{\rel}^k\circ g$, then
\begin{equation}
\label{eq-zhang-ineq}
\frac{\deg(L)}{\deg(f)}<\frac{\mu_{\min}((F_{\rel}^k)^*\tilde{E})}{p^k}+\frac{\varepsilon'}{2}\cdot\|C\|
\end{equation}
using $\deg(f)=p^k\deg(g)$, and $\mu_{\min}((F_{\rel}^k)^*\tilde{E})=\mu(Q)$. Since $f(Y)=C$, Barton's inequality $\deg(L)/\deg(f)\ge\varepsilon'\|C\|$ and \eqref{eq-zhang-ineq} together imply
$$
\frac{\mu_{\min}((F_{\rel}^k)^*\tilde{E})}{p^k}>\frac{\varepsilon'}{2}\cdot\|C\|
$$
Therefore, by letting $\varepsilon=\varepsilon'/2$, we obtain $L_{\min}(E|_C)\ge\varepsilon\|C\|$.

Conversely, suppose we have some real number $\varepsilon>0$, and inequality $L_{\min}(E|_C)\ge\varepsilon\|C\|$ for any integral, closed curve $C$ in $X$. We will check that $E$ is ample using Thm.~\ref{thm-barton}. Let $\nu$, $\tilde{C}$, and $\tilde{E}$ be as before. Indeed, given any finite morphism $f:Y\rightarrow X$ where $Y$ is a smooth projective curve, and any line bundle quotient $f^*E\rightarrow L$, there is a factorization
$$
\xymatrix{Y\ar[r]^{g\quad}\ar[dr]_f & \tilde{C}^{(-k)} \ar[d]^{\nu\circ F_{\rel}^k} \\ & X}
$$
where $g$ is a finite, separable morphism of smooth projective curves. Applying Lem.~\ref{lem-sep-hn} to the morphism $g$, we obtain
$$
\frac{\deg(L)}{\deg(f)}\ge\frac{\mu_{\min}(f^*E)}{\deg(f)}=\frac{\mu_{\min}((F^k_{\rel})^*\tilde{E})}{p^k}\ge L_{\min}(\tilde{E})=L_{\min}(E|_C)\ge\varepsilon\|C\|
$$
and the result follows.
\end{proof}

\begin{cor}[H.~Brenner \cite{Bre} and I.~Biswas \cite{Bis}]
\label{cor-ample-criterion}
Let $X$ be a smooth projective curve over an algebraically closed field $k$, and $E$ be a vector bundle over $X$. Then $E$ is ample if and only if $L_{\min}(E)>0$.
\end{cor}
\begin{proof}
This is Thm.~\ref{thm-ample-criterion-text} in dimension one.
\end{proof}

Combining our criterion for amplitude with earlier results, we find

\begin{cor}
\label{cor-deg-criterion}
Let $X$ be a smooth projective curve over an algebraically closed field $k$ of characteristic $p>0$. Let $E$ be a rank-$r$ semistable vector bundle over $X$. Suppose $r\le p$, then
\begin{enumerate}[(i)]
	\item $\deg(E)>r(r-1)(g-1)/p$ implies that $E$ is ample;
	
	\item if $E$ is maximally Frobenius-destabilized, then $\deg(E)\le r(r-1)(g-1)/p$ implies that $E$ is not ample.
\end{enumerate}
\noindent
Suppose $r>p$. Then $\deg(E)\ge r(g-1)$ implies that $E$ is ample.
\end{cor}
\begin{proof}
In the $r\le p$ case, Prop.~\ref{prop-langer-ineq} shows that $\mu(E)-L_{\min}(E)\le(r-1)(g-1)/p$, and equality is attained if and only if condition \eqref{eq-condition} holds. Since $E$ is ample if and only if $L_{\min}(E)>0$, by Cor.~\ref{cor-ample-criterion}, the statements (i) and (ii) readily follow. In the $r>p$ case, we use the sharper inequality $\mu(E)-L_{\min}(E)<g-1$ from Prop.~\ref{prop-langer-ineq-r>p}.
\end{proof}

We now use Cor.~\ref{cor-deg-criterion} to construct non-ample semistable vector bundles.

\begin{cor}[$r=p$]
Suppose $g\ge 2$. For every integer $d\le (p-1)(g-1)$, there exists a rank-$p$ vector bundle $E$ over any smooth projective curve $X$ of genus $g$ such that
\begin{enumerate}[(i)]
	\item $\deg(E)=d$, and
	\item $E$ is semistable but not ample.
\end{enumerate}
\end{cor}
\begin{proof}
Choose a line bundle $L$ over $X^{(-1)}$ of degree $d-(p-1)(g-1)$. Then $E:=(F_{\rel})_*L$ is a maximally Frobenius-destabilized vector bundle (Prop.~\ref{prop-r=p}) of rank $p$ and degree $d$. In particular, $E$ is semistable, and Cor.~\ref{cor-deg-criterion} shows that $E$ is not ample.
\end{proof}

\begin{cor}[$r<p$]
Suppose $g\ge 2$ and $p$ does not divide $g-1$; fix a natural number $r<p$. For every integer $d\le r(r-1)(g-1)/p$ which is \emph{divisible by $r$}, there exist a smooth projective curve $X$ of genus $g$ and a rank-$r$ vector bundle $E$ over $X$ such that
\begin{enumerate}[(i)]
	\item $\deg(E)=d$, and
	\item $E$ is semistable but not ample.
\end{enumerate}
\end{cor}
\begin{proof}
It follows from Prop.~\ref{prop-any-g-any-r} that there exists a genus-$g$ curve $X$, and a rank-$r$ maximally Frobenius-destabilized vector bundle $E_0$ over $X$. The construction there shows that $\deg(E_0)=0$. For some line bundle $L$ over $X$, the twist $E:=E_0\otimes L$ is of degree $d$. Since $E$ is still maximally Frobenius-destabilized, Cor.~\ref{cor-deg-criterion} shows that $E$ is not ample.
\end{proof}

\begin{rem}
In the $r<p$ case, our technique can only be used to construct non-ample semistable bundles whose degree is divisible by $r$. Indeed, the computation \eqref{eq-deg-lr} shows that the degree of \emph{any} maximally Frobenius-destabilized vector bundle (with rank $r<p$) is divisible by $r$.
\end{rem}


\medskip
\begin{thebibliography}{0}
\bibitem{Bar}
C.~M.~Barton, Tensor products of ample vector bundles in characteristic $p$, \emph{Amer.~J.~Math.} (1971): 429-438.

\bibitem{Bis}
I.~Biswas, A criterion for ample vector bundles over a curve in positive characteristic, \emph{Bull.~Sci.~Math.} 129.6 (2005): 539-543.

\bibitem{Bre}
H.~Brenner, Slopes of vector bundles on projective curves and applications to tight closure problems, \emph{Trans.~Amer.~Math.~Soc} 356.1 (2004): 371-392.

\bibitem{dJS}
A.~J.~De Jong and J.~Starr, A note on Fano manifolds whose second Chern character is positive, preprint (2005).

\bibitem{Stacks}
A.~J.~De Jong \emph{et al.}, \emph{The Stacks Project} (2014).

\bibitem{Huy}
D.~Huybrechts and M.~Lehn, \emph{The geometry of moduli spaces of sheaves} (Cambridge University Press, 2010).

\bibitem{Ila}
S.~Ilangovan, V.~B.~Mehta and A.~J.~Parameswaran, Semistability and semisimplicity in representations of low height in positive characteristic, \emph{A tribute to CS Seshadri} (Birkh\"{a}user, 2003): 271-282.

\bibitem{Jos-Xia}
K.~Joshi and E.~Z.~Xia, Moduli of vector bundles on curves in positive characteristics, \emph{Compos.~Math.} 122.03 (2000): 315-321.

\bibitem{Jos}
K.~Joshi, S.~Ramanan, E.~Xia and J.-K.~Yu, On vector bundles destabilized by Frobenius pull-back, \emph{Compos.~Math.} 142.03 (2006): 616-630.

\bibitem{Jos-Pau}
K.~Joshi and C.~Pauly, Hitchin-Mochizuki morphism, Opers and Frobenius-destabilized vector bundles over curves, preprint (2009), arXiv:0912.3602.

\bibitem{Kat}
N.~M.~Katz, Nilpotent connections and the monodromy theorem: Applications of a result of Turrittin, \emph{Publications math\'{e}matiques de l'IHES} 39.1 (1970): 175-232.

\bibitem{Kle}
S.~L.~Kleiman, Toward a numerical theory of ampleness, \emph{Ann.~of Math.} (1966): 293-344.

\bibitem{Lange}
H.~Lange and C.~Pauly, On Frobenius-destabilized rank-2 vector bundles over curves, preprint (2003), arXiv:math/0309456.

\bibitem{Lan}
A.~Langer, Semistable sheaves in positive characteristic, \emph{Ann.~of Math.} (2004): 251-276.

\bibitem{Li}
L.-G.~Li, The morphism induced by Frobenius push-forward, \emph{Sci.~China Math.} 57.1 (2014): 61-67.

\bibitem{Liu}
C.~Liu and M.~Zhou, Stable bundles as Frobenius morphism direct image, \emph{C.~R.~Math.~Acad.~Sci.~Paris} 351.9 (2013): 381-383.

\bibitem{Moc}
S.~Mochizuki, \emph{Foundations of p-adic Teichm\"{u}ller Theory} (American Mathematical Society, 1999).

\bibitem{Oss}
B.~Osserman, Frobenius-unstable bundles and $p$-curvature, \emph{Trans.~Amer.~Math.~Soc} 360.1 (2008): 273-305.

\bibitem{She}
N.~I.~Shepherd-Barron, Semi-stability and reduction mod $p$, \emph{Topology} 37.3 (1998): 659-664.

\bibitem{Sun}
X.~Sun, Direct images of bundles under Frobenius morphism, \emph{Invent.~Math.} 173.2 (2008): 427-447.

\bibitem{Wak}
Y.~Wakabayashi, An explicit formula for the generic number of dormant indigenous bundles, preprint (2013).
\end{thebibliography}
\end{document}